\documentclass[12pt,a4paper,reqno]{amsart}
\usepackage{amssymb}
\usepackage{amsmath,amsthm}
\usepackage{a4wide}
\usepackage{cancel}
\usepackage{cite}
\newtheorem{theorem}{Theorem}[section]
\newtheorem{lemma}[theorem]{Lemma}
\newtheorem{proposition}[theorem]{Proposition}

\theoremstyle{definition}

\newtheorem{example}[theorem]{Example}

\def\N{\mathbb N}

\newcommand{\R}{\mathbb{R}}

\newcommand{\al}{\alpha}

\newcommand{\NN}{\mathbb N}

\newcommand{\RR}{\mathbb R}

 \newcommand{\bea}{\begin{eqnarray}}
 \newcommand{\eea}{\end{eqnarray}}

 \newcommand{\beas}{\begin{eqnarray*}}
 \newcommand{\eeas}{\end{eqnarray*}}

\theoremstyle{remark}
\newtheorem{remark}[theorem]{Remark}

\numberwithin{equation}{section}

\theoremstyle{remark}

\usepackage{color}

\usepackage{xcolor}

\title[New Zemanian Type Spaces and the Quasiasymptotics for the FrHT]{New Zemanian Type Spaces and the Quasiasymptotics for the Fractional Hankel Transform}

\author[S. Atanasova]{Sanja Atanasova}
\address{Faculty of Electrical Engineering and Information Technologies, Ss. Cyril and Methodius University\\  Rugjer Boshkovik 18\\
1000 Skopje\\ North Macedonia}
\email{ksanja@feit.ukim.edu.mk}
\author[S. Jak\v{s}i\'{c}]{Smiljana Jak\v{s}i\'{c}}
\address{Faculty of Foresty, University of Belgrade\\  Kneza Vi\v{s}eslava 1, \\ 11030 Belgrade \\ Serbia}
\email{smiljana.jaksic@sfb.bg.ac.rs}
\author[S. Maksimovi\'{c}]{Snje\v{z}ana Maksimovi\'{c}}
\address{Faculty of Architecture, Civil Engineering and Geodesy, University of Banja Luka\\ Bulevar vojvode Petra Bojovi\'{c}a 1A \\ 78000 Banja Luka\\ Bosnia and Hercegovina} \email{snjezana.maksimovic@aggf.unibl.org}
\author[S. Pilipovi\'{c}]{Stevan Pilipovi\'{c}}
\address{Faculty of Sciences and Mathematics, University of Novi Sad\\ Trg D. Obradovica 4\\ 21000 Novi Sad\\ Serbia}
\email{pilipovic@dmi.uns.ac.rs}

\begin{document}

\begin{abstract}
We present an Abelian  theorem for the fractional Hankel transform (FrHT) on the Montel space $\mathcal{K}_{-1/2}(\RR_+)$, which is designed to overcome the limitations of the Zemanian space $\mathcal K^{\mu}(\R_+)$. The essential part of the paper is a new construction of the basic space $\mathcal{K}_{-1/2}(\RR_+)$, ensuring Montel space properties that guarantee desirable topological features such as the equivalence of weak and strong convergence. Also, we prove the continuity of the FrHT on this Montel space, both in function and distribution settings. The paper concludes with a Tauberian theorem that provides the converse implication showing that under appropriate growth and limit conditions the asymptotic behavior of the original distribution can be deduced from that of its FrHT. For this purpose a new space $\mathcal B_{-1/2}(\RR_+)$ is construct.
\end{abstract}

\keywords{Fractional Hankel transform, distributions, Abelian and Tauberian theorems}


\subjclass[2020]{40E05, 46F05, 46F12, 47B35}
\maketitle

\section{Introduction}

This paper continues the line of research initiated in our previous work \cite{AJMP}, where we investigated Abelian and Tauberian type theorems for the fractional Hankel transform (FrHT) in the framework of Zemanian distribution spaces. In that setting, we established a foundational understanding of the relationship between the quasiasymptotic behavior of distributions and the behavior of their FrHTs. However, a central open question remained: whether the quasiasymptotic behavior of a distribution $f$ necessarily implies the quasiasymptotic behavior of its image under the transform, $H^\alpha_{\mu_0}f$ for $\mu_0 \in \mathbb{N}_0$.

Our aim in this paper is to resolve this question by developing a refined functional-analytic framework in which such an implication can be  established. To overcome the  limitations of the classical Zemanian spaces $\mathcal{K}^\mu(\mathbb{R}_+)$—most notably, the lack of equivalence between weak and strong convergence—we construct a new sequence of spaces $\mathcal{K}_{-1/2,\mu}(\mathbb{R}_+)$, indexed by $\mu \in \mathbb{N}_0$, equipped with appropriate norms. The projective limit of this sequence, denoted $\mathcal{K}_{-1/2}(\mathbb{R}_+)$, is shown to be a Montel space, providing the compactness and reflexivity properties necessary for our analysis.

We recall that the fractional Hankel transform generalizes the classical Hankel transform by replacing the standard Bessel kernel with one of fractional order. This idea, introduced by Namias and developed further by Kerr \cite{5, 55} and others \cite{Zem, Zem1, Betancor, Betancor1, Hankel1, Ridenhour, She}, has proven particularly useful in the study of wave propagation, optics, and pseudo-differential operators \cite{Pras}.

In this paper, we define the FrHT $H^\alpha_{\mu_0}$ on the dual space $\mathcal{K}'_{-1/2}(\mathbb{R}_+)$ and prove an Abelian theorem which confirms that, in this refined setting, the quasiasymptotic behavior of a distribution indeed implies the corresponding behavior of its FrHT. This is in sharp contrast to the setting of $(\mathcal{K}^\mu(\mathbb{R}_+))'$, where such a result could not be obtained due to topological obstructions. Importantly, our approach ensures that the results hold uniformly for all $\mu_0 \in \mathbb{N}_0$. Additionally, we establish a Tauberian theorem by introducing a new space $\mathcal{B}_{-1/2}(\mathbb{R}_+)$, which serves as a natural generalization of the Zemanian space $\mathcal{B}^\mu(\mathbb{R}_+)$.

It is worth noting that Zemanian-type Zemanian type spaces
 $(\mathcal K^\mu(\R_+))'$ and $(\mathcal K^\nu(\R_+))',$ $ \mu, \nu \geq -1/2,$ are comparable only if $|\mu-\nu|=2n, n\in\N$. Because of that, the construction of
$\mathcal K_{-1/2}(\R_+)$ is essential for the analysis of this paper.
We also prove a Tauberian type result, introducing the additional space $\mathcal B_{-1/2}(\R_+)$ which is a generalization of the corresponding Zemanian space $\mathcal B^\mu(\R_+).$

\section{Space $\mathcal{K}_{-1/2}(\RR_+)$}\label{Sec prostor K}

Let $m,k\in\N_0$, and $\mu\geq-1/2$. Recall the definition of the Zemanian space \cite{Zem}: A smooth and complex-valued function
 $\varphi(x)$ on $\RR_+=(0,+\infty)$ belongs to the Zemanian  space $\mathcal{K}_{m,k}^{\mu}(\RR_+)$ if
$$\gamma_{m,k}^\mu(\varphi)=\sup_{x\in \RR_+}\Big|x^m(x^{-1}D)^k\big(x^{-\mu-\frac{1}{2}}\varphi(x)\big)\Big|<\infty.$$
The locally convex space  $\mathcal{K}^{\mu}(\RR_+)$ is the intersection of $\mathcal{K}_{m,k}^{\mu}(\RR_+)$:
$$\mathcal{K}^{\mu}(\RR_+)=\bigcap_{m,k\in\N_0}\mathcal{K}_{m,k}^{\mu}(\RR_+).$$
(Note that Zemanian denote this space with $\mathcal K_\mu(\R_+).$ We put the index $\mu$ as a super-index in order to exclude the confusion with the notation $\mathcal K_{-1/2}(\R_+)$ which will be given below.)
A topology of $\mathcal{K}^{\mu}(\RR_+)$ is assigned by taking $\{\gamma_{m,k}^\mu\}_{m,k=0}^\infty$ as its seminorms.
With
$$\gamma_r^\mu(\varphi)=\max_{0\leq m \leq r \atop 0\leq k\leq r}\gamma_{m,k}^\mu(\varphi),$$
one obtains norms in $\mathcal{K}^{\mu}(\RR_+)$ which generate the same topology (see \cite[(8), p.133]{Zem}). Zemanian showed that $\mathcal{K}^{\mu}(\RR_+)$ is complete, and therefore a Fr\'{e}chet space. The space $\mathcal{D}(\RR_+)$ is a subspace of $\mathcal{K}^{\mu}(\RR_+)$, it is not dense in $\mathcal{K}^{\mu}(\RR_+)$ but $\mathcal{K}^{\mu}(\RR_+)$ is a dense subspace of $\mathcal{E}(\RR_+)$.

The dual space $(\mathcal K^\mu(\RR_+))'$ consists of distribution with slow growth and it is also complete space.

We point out that the Zemanian space $\mathcal{K}^\mu(\RR_+)$ is not a Montel space and provide an example to illustrate this.
\begin{example}
We denote by $\mathcal C_0(\RR_+)$  the space of continuous functions with compact support. Let $n\in\NN$. We define the function  $\delta_{n}(t)=(8n)\theta((8n)t)$, where $t\in\R$ and $\theta$ is an even, non-negative function supported by $[-1/2,1/2]$. Furthermore, let $\{\phi_n\}_{n=0}^\infty$ be the sequence of functions from $\mathcal C_0(\RR_+)$ such that
\begin{equation}
 \phi_n(t)=\begin{cases}
 0, & t\in(0,\frac{1}{2n}),\\
 t^{2n}, & t\in(\frac{1}{2n},1-\frac{1}{4n}), \\
 1, & t\in(1-\frac{1}{4n},3)
 \end{cases}
   \end{equation}
Then the sequence
 $$ \kappa_n=\delta_n*(t^{\mu+1/2}\phi_n)$$
is bounded in $\mathcal K^\mu(\RR_+)$, but it does not have a convergent subsequence in
$\mathcal K^\mu(\RR_+)$.
\end{example}

The next step is to introduce new seminorms that define a Montel space suitable for describing the Abelian  result discussed in Section \ref{Sec Abelian and Tauberian-type results}.

For each $\mu\in\NN_0$, we define the space $\mathcal{K}_{-1/2,\mu}(\RR_+)$ of smooth, complex-valued functions $\varphi(x)$ on $\RR_+$ such that for each pair of $m,k\in\NN_0$
\begin{equation}\label{2dana}
\beta^\mu(\varphi)=\sup_{0\leq m\leq2\mu \atop 0\leq k\leq2\mu}\beta^\mu_{m,k}(\varphi)<\infty,
\end{equation}
where
\begin{equation}\label{1dan}
\beta^\mu_{m,k}(\varphi)=\gamma^{2\mu-1/2}_{m,k}(\varphi)+\gamma^{2\mu+1/2}_{m,k}(\varphi).
\end{equation}

Now, we need to verify  that
\begin{equation}\label{3dana}
\beta^\mu(\varphi)\leq\beta^{\mu+1}(\varphi), \;\mu\in\N_0, \;\varphi\in\mathcal K_{-1/2,\mu}(\RR_+).
\end{equation}
By (\cite[Lemma 1]{Zem1}), there holds
\begin{equation}\label{nej1}
\gamma^{2\mu - 1/2}_{2\mu,2\mu}(\varphi)\leq 4\mu\gamma^{2(\mu+1)-1/2}_{2\mu,2\mu-1}(\varphi)+\gamma^{2(\mu+1)-1/2}_{2\mu+2,2\mu}(\varphi)
\end{equation}
and
\begin{equation}\label{nej2}
\gamma^{2\mu + 1/2}_{2\mu,2\mu}(\varphi)\leq 4\mu\gamma^{2(\mu+1)+1/2}_{2\mu,2\mu-1}(\varphi)+\gamma^{2(\mu+1)+1/2}_{2\mu+2,2\mu}(\varphi).
\end{equation}
Using the inequalities (\ref{nej1}) and (\ref{nej2}), we obtain
$$ \beta^\mu_{2\mu,2\mu}(\varphi)\leq 4\mu\beta^{\mu+1}_{2(\mu+1),2(\mu+1)}(\varphi)+\beta^{\mu+1}_{2(\mu+1),2(\mu+1)}(\varphi)\leq(4\mu+1)\beta^{\mu+1}_{2(\mu+1),2(\mu+1)}(\varphi).
$$
Hence, (\ref{3dana}) follows.

Finally, the locally convex space $\mathcal{K}_{-1/2}(\RR_+)$ is defined to be the intersection of $\mathcal{K}_{-1/2,\mu}(\RR_+)$, $\mu\in\NN_0$:
$$ \mathcal{K}_{-1/2}(\R_+)=\bigcap_{\mu\in\N_0}\mathcal{K}_{-1/2,\mu}(\RR_+)=\{\phi\in
C^\infty(\R_+): \beta^\mu(\phi)<\infty, \mu\in\N_0\}.$$
The space $\mathcal{K}_{-1/2}(\R_+)$ is endowed with the projective limit topology.  In view of (\ref{2dana}) and (\ref{1dan}), the  inclusion mappings
$$\mathcal K_{-1/2}(\RR_+)\rightarrow \mathcal K_{2\mu,2\mu}^{2\mu-1/2}(\RR_+)\;\mbox{ and }\;\mathcal K_{-1/2}(\RR_+)\rightarrow \mathcal K_{2\mu,2\mu}^{2\mu+1/2}(\RR_+),  \ \ \mu\in\N_0, \;$$
\mbox{are continuous}.
The properties of $\mathcal K_{-1/2}(\RR_+)$ follow from the corresponding ones for the Zemanian spaces;  $\mathcal{D}(\RR_+)\subset \mathcal{K}_{-1/2}(\RR_+)$ and the elements of $\mathcal{K}_{-1/2}(\RR_+)$ rapidly decrease to zero as $x\rightarrow+\infty$. Furthermore, $\mathcal{K}_{-1/2}(\RR_+)$ is a dense subspace of $\mathcal{E}(\RR_+)$ and $\mathcal{E}'(\RR_+)\subset \mathcal{K}_{-1/2}'(\RR_+)$ ($\mathcal{K}_{-1/2}'(\RR_+)$ is the dual space of $\mathcal{K}_{-1/2}(\RR_+)$). The $\mathcal{K}_{-1/2}'(\RR_+)$  is equipped with the strong dual topology.
For every $f\in \mathcal K_{-1/2}'(\RR_+)$ there exist a positive constant $C$ and a non-negative integer $\mu$ such that
\begin{equation}\label{cous}
|\langle f,\varphi\rangle|\leq C\beta^\mu(\varphi),\quad\mbox{for all}\quad \varphi\in\mathcal{K}_{-1/2}(\RR_+).
\end{equation}

The proof of the following proposition relies on the completeness of the Zemanian space $\mathcal{K}^{\mu}(\RR_+)$.
\begin{proposition}
$\mathcal{K}_{-1/2}(\RR_+)$  is a Fr\'{e}chet space.
\end{proposition}
\noindent This proposition implies  that $\mathcal{K}_{-1/2}(\R_+)$ is a barreled space.

\begin{proposition}\label{propmont}
A sequence $\{\varphi_\nu\}_{\nu\in\NN_0}$ in $\mathcal{K}_{-1/2}(\RR_+)$ converges to zero if and only if it is bounded and for every compact set $K\subset \RR_+$, and every $\mu\in\NN_0,$
\begin{equation}\label{sup1}
\sup_{x\in K \atop 0\leq m\leq 2\mu, 0\leq k\leq 2\mu}\Big|x^m(x^{-1}D)^k\big(x^{-2\mu}\varphi_\nu(x)\big)\Big|\rightarrow 0, \quad \nu\rightarrow\infty, \mu\in\N_0
\end{equation}
 and
\begin{equation}\label{sup2}
\sup_{x\in K \atop 0\leq m\leq 2\mu, 0\leq k\leq 2\mu}\Big|x^m(x^{-1}D)^k\big(x^{-2\mu-1}\varphi_\nu(x)\big)\Big|\rightarrow 0, \quad \nu\rightarrow\infty,  \mu\in\N_0.
\end{equation}

\end{proposition}
\begin{proof}
We give a proof for (\ref{sup1}), since the proof for (\ref{sup2}) is similar.
For every $\mu\in\NN_0$ and every compact set $K\subset \RR_+$,
\begin{equation}\label{*}
\sup_{x\in K \atop 0\leq m\leq 2\mu, 0\leq k\leq 2\mu}x^m|(x^{-1}D)^k\big(x^{-2\mu}\varphi_\nu(x)\big)|\rightarrow0 \ \ \text{ as } \ \  \nu\rightarrow\infty
\end{equation}
if and only if for every $\alpha\in\N_0$
$$\sup_{x\in K}|\varphi_\nu^{(\alpha)}(x)|\rightarrow 0 \ \text{ as }  \ \ \nu\rightarrow\infty.
$$

\noindent It is easy to see that if $\{\varphi_\nu\}_{\nu\in\NN_0}$ converges to zero in $\mathcal{K}_{-1/2}(\RR_+)$, it is bounded and \eqref{*} holds. Let us prove the opposite direction. For this, it is enough to prove that for all $\mu\in\N_0$ and all $\delta>0$ there exist $\varepsilon>0$, $M>1$ and $\nu_0\in\N$ such that
\begin{equation}\label{(1)}
\sup_{x\in (0,
\varepsilon) \atop 0\leq m\leq 2\mu, 0\leq k\leq 2\mu}x^m|(x^{-1}D)^k\big(x^{-2\mu}\varphi_\nu(x)\big)|<\delta \ \ \text{ for } \ \ \nu>\nu_0
\end{equation}
\noindent and
\begin{equation}\label{(2)}
\sup_{x \geq M \atop 0\leq m\leq 2\mu, 0\leq k\leq 2\mu}x^m|(x^{-1}D)^k\big(x^{-2\mu}\varphi_\nu(x)\big)|<\delta \ \ \text{ for } \ \ \nu\rightarrow\nu_0.
\end{equation}
Firstly, we show the estimate \eqref{(2)}. Since there exists $M>1$ such that

\begin{eqnarray*}
& & \sup_{x\geq M \atop 0\leq m\leq 2\mu, 0\leq k\leq 2\mu}x^m|(x^{-1}D)^k\big(x^{-2\mu}\varphi_\nu(x)\big)|\\
& \leq & \frac{1}{M}\sup_{x\geq M\atop  m+1 \leq 2\mu, \ k+1 \leq 2\mu}x^{m+1}|(x^{-1}D)^k(x^{-2(\mu+1)}\varphi_\nu(x))|
)\leq  \frac{1}{M}
\beta^{2(\mu+1)}(\varphi_\nu), \; \nu\in\N,
\end{eqnarray*}
the proof of \eqref{(2)} follows.

\noindent The next step is to show the estimate \eqref{(1)}.
Using \cite[p.132]{Zem}, for $m,k,\mu\in\N_0$ and $\varphi\in \mathcal{K}_{-1/2}(\mathbb R_+)\subset\mathcal K^{\mu}(\mathbb R_+)$, the following recursive relation holds
$$ x^m(x^{-1}D)^k\big(x^{-2\mu}\varphi(x)\big)=2kx^m(x^{-1}D)^{k-1}\big(x^{-2\mu-2}\varphi(x)\big)+x^{m+2}(x^{-1}D)^k\big(x^{-2\mu-2}\varphi(x)\big).$$
 This implies
\begin{align*}
x^m(x^{-1}D)^k\big(x^{-2\mu}\varphi(x)\big)&=2k(2k-2)x^m(x^{-1}D)^{k-2}\big(x^{-2\mu-4}\varphi(x)\big)\\
&+2kx^{m+2}(x^{-1}D)^{k-1}\big(x^{-2\mu-4}\varphi(x)\big)+x^{m+2}(x^{-1}D)^k\big(x^{-2\mu-2}\varphi(x)\big).
\end{align*}
We conclude that
\begin{eqnarray}
 & & x^m(x^{-1}D)^k\big(x^{-2\mu}\varphi(x)\big)\nonumber \\
 & = & (2k)!!x^m\big(x^{-2\mu-2k}\varphi(x)\big)+x^{m+2}\sum_{j=0}^{k-1}C_j(x^{-1}D)^{k-j}\big(x^{-2\mu-2(j+1)}\varphi(x)\big), \label{trecidan}
\end{eqnarray}

\noindent where $C_0=1, \ C_1=2k, \ C_2=2k(2k-2), \ \ldots, \ C_{k-1}=2k(2k-2)\cdots2$.

\noindent By (\ref{trecidan}) and assumptions that  $m\leq 2\mu, \ k\leq 2\mu$ we have

\begin{eqnarray*}
& & \sup_{x\in (0,
\varepsilon) \atop 0\leq m\leq 2\mu, 0\leq k\leq 2\mu}x^m|(x^{-1}D)^k\big(x^{-2\mu}\varphi_\nu(x)\big)|\\
& \leq &  (2\mu)!!
\varepsilon\sup_{x\in(0,\varepsilon)}(x^{-2\mu-2k}|\varphi(x)|+
\varepsilon x^{2}\sum_{j=0}^{k-1}C_j(x^{-1}D)^{k-j}(x^{-2\mu-2(j+1)}|\varphi(x)|)\\
& \leq & \varepsilon(\gamma^{4\mu}_{2\mu,2\mu}(\varphi)+
\varepsilon C(\mu)\gamma^{4\mu}_{2\mu,2\mu}(\varphi))\leq\beta^{2\mu}(\varphi).
\end{eqnarray*}
This implies the assertion.
\end{proof}

\begin{theorem}
The space $\mathcal{K}_{-1/2}(\RR_+)$ is a Montel space.
\end{theorem}
\begin{proof}
 We need to prove the Heine-Borel property, i.e. that every bounded set $A\subset\mathcal{K}_{-1/2}(\RR_+)$ is relatively compact. Using Proposition \ref{propmont} it suffices to prove that for every compact set $K$, a sequence $\{\varphi_\nu\}_{\nu\in\NN_0}\in A$ has a convergent sub-sequence again denoted by $\{\varphi_\nu\}_{\nu\in\NN_0}$ that has the property that for every $\alpha\in\N_0$, $\sup_{x\in K}|\varphi_\nu^{(\alpha)}(x)|\rightarrow 0$  as $\nu\rightarrow\infty$.

Let  $\{\varphi_\nu\}_{\nu\in\NN_0}$ be a bounded sequence in $A$. Since $\mathcal{K}_{-1/2}(\RR_+)$ is a  barrelled space, it follows that  $\{\varphi_\nu\}_{\nu\in\NN_0}$  is uniformly bounded, as well as the sequences  with  derivatives of all orders.  This implies  that for each $\alpha$, $\{\varphi_\nu^{(\alpha)}\}_{\nu\in\NN_0}$ is a uniformly bounded, equicontinuous set, and thus by the Arzel\`{a} -Ascoli theorem, it has a uniformly convergent subsequence. This completes the proof.
\end{proof}

\begin{remark}\label{jaka=slaba konv}
Recall, that every Montel space is reflexive. Moreover, it should be noted that the strong dual of a Montel space is also a Montel space. Finally, it is crucial to mention that the weak and the strong convergence in a Montel space are equivalent, \cite{horvat}.
\end{remark}

\section{Fractional Hankel transform}\label{FRHT}

The Hankel transform of order $\mu\geq-1/2,$ of a function $f\in L^1(\RR_+)$ is
defined by
\begin{equation*}\label{ht}
H_\mu f(\xi)=\int_0^\infty \sqrt{x\xi}J_\mu(x\xi) f(x)dx, \ \ \xi\in \RR_+,
\end{equation*}
where,
$J_\mu(x)$
is the Bessel function of first kind and order $\mu$. If $H_{\mu} f\in L^1(\RR_+)$, then the inverse Hankel transform is given by
$ f(x)=\int_0^\infty \sqrt{x\xi}J_{\mu}(x\xi) H_{\mu} f(\xi)d\xi, \ \ x\in \RR_+.$

Recall \cite[(1.3)]{Pras}, the FrHT with parameter $\alpha$ of $f(x)$ for $\mu\geq -1/2$ and $\alpha\in(0,\pi)$,  is defined by
$
{H}_\mu^\alpha f(\xi)=
\int_{0}^\infty K_{\alpha}(x,\xi)f(x)dx,
$
where
\begin{equation*}\label{kernel}
 K_{\alpha}(x,\xi)=\begin{cases}
 C_{\alpha,\mu} e^{-i(\frac{x^2+\xi^2}{2}c_1)}\sqrt{x\xi c_2}J_\mu(x\xi c_2), & \alpha\neq\frac{\pi}{2},\\
 \sqrt{x\xi}J_\mu(x\xi), & \alpha=\frac{\pi}{2}
 \end{cases}
   \end{equation*}
is the kernel of the FrHT,
 $c_1=\cot\alpha$, $c_2=\csc \alpha$,  $C_{\alpha,\mu}=\frac{e^{i(1+\mu)(\frac{\pi}{2}-\alpha)}}{\sin\alpha}$.
By\cite{Pras}, the FrHT is a continuous linear mapping of $\mathcal K^\mu(\R_+)$ onto itself.

For $f\in \mathcal K_{-1/2}(\RR_+)$, and $\mu\in\N_0,$ the inverse FrHT is defined in \cite[(1.4)]{5} as
\begin{equation}\label{inverse}
(H_{\mu}^{\alpha})^{-1}H_{\mu}^{\alpha}f(x)=f(x)=\int_{0}^\infty \overline{K}_{\alpha}(x,\xi)(H_{\mu}^\alpha f)(\xi)d\xi,
\end{equation}
where
$\overline{K}_{\alpha}(x,\xi)=C_{\alpha, \mu}^{\star}e^{i(\frac{x^2+\xi^2}{2}c_1)}\sqrt{x\xi c_2}J_{\mu}(x\xi c_2), \, C_{\alpha,\mu}^{\star}=\overline{C_{\alpha,\mu}}\sin{\alpha}.$
In our results, we benefit from the  following formula  (\cite[Theorem 4.23]{5})\
\begin{equation}\label{FrHT veza HT}
H_{\mu}^\alpha f(\xi)=C_{\alpha,\mu}e^{-ic_1\xi^2/2}H_{\mu}(e^{-ic_1 x^2/2}f(x))(c_2\xi), \ x,\xi\in \RR_+,
\end{equation}
which gives the relation between the FrHT and the classical Hankel transform.

We fix $\mu_0 \in \NN_0$ for the rest of the paper.

 \begin{theorem}   The FrHT $H^\alpha_{\mu_0}$ is a continuous linear mapping of $\mathcal K_{-1/2}(\R_+)$ onto itself.
\end{theorem}
\begin{proof}
Linearity and surjectivity follow straightforward. For continuity, it is enough to show that for every sequence $\{\varphi_\nu\}_{\nu\in\NN_0}$ such  that $\varphi_\nu\rightarrow 0$ in $\mathcal K_{-1/2}(\RR_+)$, it holds that $H_{\mu_0}^\alpha(\varphi_\nu)\rightarrow 0$ in $\mathcal K_{-1/2}(\RR_+)$, when $\nu\rightarrow\infty$.

\noindent By (\ref{FrHT veza HT}), $\Phi_{\mu_0}\varphi(\xi)=H_{\mu_0}(e^{-i \xi^2 c_1/2}\varphi)(\xi c_2)$, $\xi\in\RR_+$, we obtain
\begin{align*}
&\sup_{\xi\in\RR_+ \atop 0\leq m \leq 2\mu, 0\leq k\leq 2\mu}|\xi^m(\xi^{-1}D)^{k}(\xi^{-2\mu}H_{\mu_0}^\alpha\varphi_\nu(\xi))|\\
&=\sup_{\xi\in\RR_+ \atop 0\leq m \leq 2\mu, 0\leq k\leq 2\mu}|\xi^m(\xi^{-1}D)^{k}(\xi^{-2\mu}C_{\alpha, \mu_0}e^{-i \xi ^2 c_1/2}\Phi_{\mu_0}\varphi_\nu(\xi))|\\
&=\sup_{\xi\in\RR_+ \atop 0\leq m \leq 2\mu, 0\leq k\leq 2\mu}\Big|\xi^m C_{\alpha,\mu_0}\sum_{k'=0}^{k}\binom{k}{k'}(-ic_1)^{k'}e^{-i \xi^2 c_1/2}(\xi^{-1}D)^{k-k'}(\xi^{-2\mu}\Phi_{\mu_0}\varphi_\nu(\xi))\Big|\\
&=C'_{\alpha,\mu_0}\sum_{k'=0}^{k}\binom{k}{k'}\sup_{\xi\in\RR_+ \atop 0\leq m \leq 2\mu, 0\leq k\leq 2\mu}|\xi^m (\xi^{-1}D)^{k-k'}(\xi^{-2\mu}\Phi_{\mu_0}\varphi_\nu(\xi))|,
\end{align*}
where $C'_{\alpha,\mu_0}=C_{\alpha,\mu_0}c_1^{k'}$.
In the similar way, it can be shown that
\begin{eqnarray*}
& & \sup_{\xi\in\RR_+ \atop 0\leq m \leq 2\mu, 0\leq k\leq 2\mu}|\xi^m(\xi^{-1}D)^{k}(\xi^{-2\mu-1}H_{\mu_0}^\alpha\varphi_\nu(\xi))|\\
& = & C'_{\alpha,\mu_0}\sum_{k'=0}^{k}\binom{k}{k'}\sup_{\xi\in\RR_+ \atop 0\leq m \leq 2\mu, 0\leq k\leq 2\mu}|\xi^m (\xi^{-1}D)^{k-k'}(\xi^{-2\mu-1}\Phi_{\mu_0}\varphi_\nu(\xi))|.
\end{eqnarray*}
\end{proof}
Using the result of \cite[Definition 3.1]{Pras}, the FrHT can be extended to the space of distributions $\mathcal K'_{-1/2}(\R_+)$.
\begin{equation*}\label{frdirect}
\langle H_{\mu_0}^{\alpha} f,\varphi\rangle=\langle  f,H_{\mu_0}^{\alpha}\varphi\rangle, \quad \varphi\in\mathcal K_{-1/2}(\RR_+).
\end{equation*}
Clearly, it is a continuous mapping of  $\mathcal K'_{-1/2}(\R_+)$ onto itself.


\section{The quasi-asymptotic behavior and quasi-asymptotic boundedness of distributions}

We recall some definitions from \cite{VDZ} and \cite{7} (see
also \cite{Vindas1,Vindas2}). A positive real-valued function, measurable on an
interval $(0,A]$ (resp. $[A,+\infty )),$ $A>0$, is called \textit{a
slowly varying function} at the origin (resp. at infinity), if
\begin{equation} \label{limL1}
\lim_{\varepsilon \to 0^{+} } \frac{L(a\varepsilon
)}{L(\varepsilon )} =1\quad \Big(\ {\rm resp.} \lim_{\lambda \to
+\infty } \frac{L(a\lambda )}{L(\lambda )} =1\Big)\quad {\rm for \ \
each} \ a>0.
\end{equation}
Throughout the rest of the article, $L$ will consistently denote a slowly varying function at the origin (resp. at infinity).

Let $f\in ({\mathcal K^{\mu}}(\RR_+))'$. It is said that the distribution $f$ has
\textit{the quasi-asymptotic behavior}  of
degree $m \in {\Bbb R}$  at zero (respectively, at $\infty$)
with respect to a slowly varying function $L$  at zero (respectively, at $\infty$)  in $({\mathcal K^{\mu}}(\RR_+))'$ if there exists $u\in ({\mathcal K^{\mu}}(\RR_+))'$
such that for every test function $\varphi \in \mathcal{K}^{\mu}(\RR_+)$ the following limit holds:
$$
\lim_{\varepsilon \to 0^{+} } \Big\langle \frac{f(x_{0} +\varepsilon
x)}{\varepsilon ^{m } L(\varepsilon )} ,\ \varphi (x)\Big\rangle
=\langle u(x),\varphi (x)\rangle\; \mbox{ respectively, } \lim_{\lambda \to \infty } \Big\langle \frac{f(\lambda
x)}{\lambda ^{m } L(\lambda )} ,\ \varphi (x)\Big\rangle
=\langle u(x),\varphi (x)\rangle .
$$

\noindent The convenient notation for the
quasi-asymptotic behavior of degree $m \in {\Bbb R}$ at a zero (respectively,  at infinity)
with respect to $L$ in $\mathcal{K}_{-1/2}'(\RR_+)$, is
\[f(x_{0} +\varepsilon x) \sim \varepsilon ^{m } L(\varepsilon )u(x)\quad {\rm as} \quad \varepsilon \to 0^{+} ,\quad \Big({\rm resp.} \quad f(\lambda x) \sim \lambda ^{m }
L(\lambda )u(x) \quad {\rm as} \quad \lambda \to +\infty\Big).\]

Note that $u$ is homogeneous with degree of homogeneity $m $, i.e.,
$u(ax)=a^{m } u(x)$, for all $a>0 $, see \cite{7,
VDZ}.

We introduce an additional concept from quasi-asymptotic analysis, namely the notion of {quasi-asymptotic boundedness},

    Let $f\in \mathcal{K}_{-1/2}'(\RR_+)$. The distribution $f$ is \emph{quasi-asymptotically
bounded} of degree $m\in \R$ at  zero with respect to $L$ in $(\mathcal{K}^{\mu}(\RR_+))'$ if

$$|\langle
f(\varepsilon x),\,\varphi(x)\rangle|\leq C\varepsilon^m
L(\varepsilon), \;\varphi\in\mathcal K_{-1/2}\;(C \mbox{ depends on } \varphi).$$

The notion of quasi-asymptotic boundedness of degree $m\in \R$ at infinity with respect to $L$ in $(\mathcal{K}^{\mu}(\RR_+))'$ is defined in an analogous manner.

\subsection{Abelian and Tauberian-type results}\label{Sec Abelian and Tauberian-type results}

In this section, we present Abelian and Tauberian-type theorems for the FrHT within the framework of the specially constructed space $\mathcal{K}_{-1/2}'(\RR_+)$, which was introduced in Section \ref{Sec prostor K} in order to obtain the desired results. Furthermore, we investigate quasi-asymptotic boundedness in the space $\mathcal{K}_{-1/2}'(\RR_+)$.
We  note that if $f\in\mathcal (K^{\mu_0}(\RR_+))'$ for some fixed $\mu_0\in\N_0$, then its restriction to $\mathcal K_{-1/2}(\RR_+)$ is an element of $\mathcal K'_{-1/2}(\RR_+)$.

The next Abelian theorem for the FrHT concerns the behavior at the point zero. An analogous result holds for the behavior at infinity; however we will not state it here, as it follows by repeating  the arguments of the forthcoming theorem.\\

Before we state the Abelian theorem, we first need the following lemma.
\begin{lemma}
Let $\varphi\in\mathcal K_{-1/2}(\R_+).$ Then:
\begin{itemize}
\item [$(a)$] $e^{ic_1 (\frac{\varepsilon x}{c_2})^2/2}\varphi(x)\in\mathcal K_{-1/2}(\R_+).$

\item [$(b)$] For every $g\in\mathcal K'_{-1/2}(\R_+)$, it holds that
$$\langle g(x),(e^{ic_1 (\frac{\varepsilon x}{c_2})^2/2}-1)\varphi(x)\rangle\rightarrow 0,\;  \ \varepsilon\rightarrow 0.
$$
\end{itemize}
\end{lemma} \label{dodlem}

\begin{proof}
(a) Let
$\Phi(x)=e^{ic_1 (\frac{\varepsilon x}{c_2})^2/2}\varphi(x) $. The action of operator $(x^{-1}D)$ over $\Phi$, with suitable $C$, gives
$$ (x^{-1}D)(e^{ic_1 (\frac{\varepsilon x}{c_2})^2/2}\varphi(x))=C\varepsilon (e^{ic_1(\frac{\varepsilon x}{c_2})^2/2})\varphi(x)+ (e^{ic_1(\frac{\varepsilon x}{c_2})^2/2})(x^{-1}D)\varphi(x).
$$
This implies that the action of $(x^{-1}D)$ lives $\Phi$ in $\mathcal K_{-1/2}(\R_+).$ It is clear that the multiplication by $x$ leaves $\Phi$ in $\mathcal K_{-1/2}(\R_+)$, i.e., $x\Phi(x)\in\mathcal K_{-1/2}(\R_+)$. This proves (a).

\noindent (b) By (\ref{cous}),
\begin{equation}\label{dodbeta}|\langle g(x),(e^{ic_1 (\frac{\varepsilon x}{c_2})^2/2}-1)\varphi(x)\rangle|\leq
C\beta^{\mu_0}((e^{ic_1 (\frac{\varepsilon x}{c_2})^2/2}-1)\varphi(x)) \;\mbox{ for some } \; C>0 \mbox{ and }\; \mu_0\in\N_0.
\end{equation}
In the calculation of the norm $\beta^{\mu_0}$, each term in the sum includes a multiplicative factor
$(e^{ic_1 (\frac{\varepsilon x}{c_2})^2/2}-1)'=C_1(\varepsilon x)
(e^{ic_1 (\frac{\varepsilon x}{c_2})^2/2})$, $C_1$ is a constant, or  $(e^{ic_1 (\frac{\varepsilon x}{c_2})^2/2}-1)$. This implies that the right hand side of (\ref{dodbeta}) tends to zero as
$\varepsilon\rightarrow 0.$ This finishes the proof of part (b).
\end{proof}

\begin{theorem}\label {te0}
Let $f\in \mathcal{K}_{-1/2}'(\R^+)$. Assume that $f$ has the quasi-asymptotic behavior of degree $m\in\RR$ at origin with respect to $L$ in  $\mathcal K_{-1/2}'(\RR_+)$, i.e.,
\begin{equation}\label{newq}
f(\varepsilon x)\sim {\varepsilon^mL(\varepsilon)} u(x) \ \ \text{ as } \ \ \varepsilon\rightarrow0^+.
\end{equation}
Then, with fixed $\mu_0\in\N_0,$
\begin{equation}\label{1newq}
{e^{ic_1 (\frac{\xi}{\varepsilon})^2/2} H_{\mu_0}^{\alpha} f \Big(\frac{\xi}{\varepsilon}\Big)}\sim \frac{C_{\alpha,\mu_0}}{c_2^{m+1}}{\varepsilon^{m+1}L(\varepsilon)}H_{\mu_0} u(\xi) \; \text{ as } \, \varepsilon\rightarrow0^+ \, \text{ in } \, \mathcal{K}_{-1/2}'(\R_+).
\end{equation}
\end{theorem}\label{th1}

\begin{proof}

Let $\varphi\in \mathcal{K}_{-1/2}(\R_+)$. By the definition, it follows

\begin{align*}&
\Big\langle e^{ic_1 (\frac{\xi}{\varepsilon})^2/2} H_{\mu_0}^{\alpha} f \Big(\frac{\xi}{\varepsilon}\Big),\varphi(\xi)\Big\rangle =\varepsilon\langle  H_{\mu_0}^{\alpha} f (\xi),e^{ic_1 \xi^2/2}\varphi(\varepsilon\xi)\rangle
=\varepsilon\langle   f (t),H_{\mu_0}^{\alpha}(e^{ic_1 \xi^2/2}\varphi(\varepsilon\xi))(t)\rangle\\
&=\varepsilon C_{\al,\mu_0}\Big\langle   f (t),\int_0^\infty e^{ic_1 \xi^2/2}\varphi(\varepsilon\xi)e^{-ic_1\frac{\xi^2+t^2}{2}}\sqrt{t\xi c_2}J_{\mu_0}(t\xi c_2)d\xi\Big\rangle\\
&=\varepsilon C_{\al,\mu_0}\Big\langle   f (t)e^{-ic_1 t^2/2},\int_{0}^\infty \varphi(\varepsilon\xi)\sqrt{t\xi c_2}J_{\mu_0}(t\xi c_2)d\xi\Big\rangle=\varepsilon C_{\al,\mu_0}\langle   f (\varepsilon x)e^{-ic_1 (\varepsilon x)^2/2},H_{\mu_0} \varphi (c_2 x)\rangle.
\end{align*}
Then,
\begin{align*}
\notag
&\Big\langle\frac{e^{ic_1 (\frac{\xi}{\varepsilon})^2/2} H_{\mu_0}^{\alpha} f (\frac{\xi}{\varepsilon})}{\varepsilon^{m+1}L(\varepsilon)},\varphi(\xi)\Big\rangle=
C_{\al,\mu_0}\Big\langle   \frac{f (\varepsilon x)e^{-ic_1 (\varepsilon x)^2/2}}{\varepsilon^m L(\varepsilon)},H_{\mu_0} \varphi (c_2 x)\Big\rangle\\&=\frac{C_{\al,\mu_0}}{c_2}\Big\langle   \frac{f (\frac{\varepsilon x}{c_2})e^{-ic_1 (\frac{\varepsilon x}{c_2})^2/2}}{\varepsilon^m L(\varepsilon)},H_{\mu_0} \varphi (x)\Big\rangle=\frac{C_{\al,\mu_0}}{c_2^{m+1}}\Big\langle   \frac{f (\frac{\varepsilon x}{c_2})e^{-ic_1 (\frac{\varepsilon x}{c_2})^2/2}}{(\frac{\varepsilon}{c_2})^m L(\frac{\varepsilon}{c_2})\frac{L(\varepsilon)}{L(\frac{\varepsilon}{c_2})}},H_{\mu_0} \varphi (x)\Big\rangle .
\end{align*}
By (\ref{newq}), the set
\begin{equation}\label{bset}\big\{  \frac{f(\frac{t x}{c_2})e^{-ic_1 (\frac{t x}{c_2})^2/2}}{(\frac{t}{c_2})^m L(\frac{t}{c_2})\frac{L(t)}{L(\frac{t}{c_2})}}: \ t\in(0,t_0)\big\}
\end{equation}
is bounded in $\mathcal K'_{-1/2}(\R_+)$ for enough small $t_0\in(0,1)$. Since
$\{\frac{L(t)}{L(\frac{t}{c_2})}: \ t\in(0,t_1)\}, (t_1<1)$, is
 a bounded set in $\R_+),$ we exclude this quotient from (\ref{bset}) and denote  by $B$ the corresponding bounded set (maybe with the smaller $t_0$) in $\mathcal K'_{-1/2}(\R_+):$
$$B=\big\{\frac{f(\frac{t x}{c_2})}{(\frac{t}{c_2})^m}: \ t\in(0,t_0)\big\}.
$$
\noindent The application of the inversion formula \eqref{inverse} (see also \eqref{FrHT veza HT}), together with the estimate \eqref{limL1}, yields

\begin{align}\label{lim}
& & \lim_{\varepsilon\rightarrow0^+}\Big\langle\frac{e^{ic_1 (\frac{\xi}{\varepsilon})^2/2} H_{\mu_0}^{\alpha} f (\frac{\xi}{\varepsilon})}{\varepsilon^{m+1}L(\varepsilon)},\varphi(\xi)\Big\rangle
 =  \frac{C_{\al,\mu_0}}{c_2^{m+1}}\lim_{\varepsilon\rightarrow0^+}\Big\langle   \frac{f (\frac{\varepsilon x}{c_2})e^{-ic_1 (\frac{\varepsilon x}{c_2})^2/2}}{(\frac{\varepsilon}{c_2})^m L(\frac{\varepsilon}{c_2})},H_{\mu_0} \varphi (x)\Big\rangle \nonumber \\
 &=& \frac{C_{\al,\mu_0}}{c_2^{m+1}}\lim_{\varepsilon\rightarrow0^+}\Big\langle   \frac{f (\frac{\varepsilon x}{c_2})}{(\frac{\varepsilon}{c_2})^m L(\frac{\varepsilon}{c_2})},H_{\mu_0} \varphi (x)\Big\rangle+\frac{C_{\al,\mu_0}}{c_2^{m+1}}\lim_{\varepsilon\rightarrow0^+}\Big\langle   \frac{f (\frac{\varepsilon x}{c_2})\big(e^{-ic_1 (\frac{\varepsilon x}{c_2})^2/2}-1\big)}{(\frac{\varepsilon}{c_2})^m L(\frac{\varepsilon}{c_2})},H_{\mu_0} \varphi (x)\Big\rangle.
\end{align}
%
%
\noindent Now we use the fact that $\mathcal K_{-1/2}(\R_+)$ is a Montel space, so strong and weak convergence in $\mathcal{K}_{-1/2}(\RR_+)$ are equivalent (cf. Remark \ref{jaka=slaba konv}). Recall, that
 $H_{\mu_0}\varphi\in\mathcal K_{-1/2}(\R_+)$. Therefore, by Lemma \ref{dodlem}, we have that for every $g\in\mathcal K'_{-1/2}(\R_+)$,  the following holds
$$
\sup_{g\in B}\langle g(x),(e^{-ic_1 (\frac{\varepsilon x}{c_2})^2/2}-1)
H_{\mu_0}\varphi(x)\rangle\rightarrow 0, \mbox{ as } \;\varepsilon\rightarrow 0.
$$
If we put $t\in(0,\varepsilon)$ in $B$, where $\varepsilon<t_0$, we obtain
$$\lim_{\varepsilon\rightarrow0^+}
\langle \frac{f(\frac{\varepsilon x}{c_2})}{(\frac{\varepsilon}{c_2})^m},(e^{-ic_1 (\frac{\varepsilon x}{c_2})^2/2}-1) H_\mu\varphi(x)\rangle=0.$$
So, we conclude that the second term on the right-hand side of (\ref{lim}) is equal to zero.

\noindent Thus, \eqref{newq} gives

\begin{eqnarray*}
& & \lim_{\varepsilon\rightarrow0^+}\Big\langle\frac{e^{ic_1 (\frac{\xi}{\varepsilon})^2/2} H_{\mu_0}^{\alpha} f (\frac{\xi}{\varepsilon})}{\varepsilon^{m+1}L(\varepsilon)},\varphi(\xi)\Big\rangle=
\frac{C_{\alpha,\mu_0}}{c_2^{m+1}}\lim_{\varepsilon\rightarrow0^+}\Big\langle\frac{f(\varepsilon x)}{\varepsilon^mL(\varepsilon)},H_{\mu_0} \varphi(x)\Big\rangle\\
& = &\frac{C_{\alpha,\mu_0}}{c_2^{m+1}}\langle u(x),H_{\mu_0} \varphi (x)\rangle=\frac{C_{\alpha,\mu_0}}{c_2^{m+1}}\langle H_{\mu_0} u(\xi), \varphi (\xi)\rangle.
\end{eqnarray*}
%
%
This proves the theorem.
\end{proof}

Before we state the Tauberian theorem for the FrHT, in a manner similar to \cite[Section 2]{Zem2}, we define the space $\mathcal B_{\mu,b}(\RR_+), \ \mu\in\N_0.$ Let $b$ be a positive real number. The space $\mathcal B_{\mu,b}(\RR_+)$ consists of all smooth complex-valued functions $\psi(x), \  x\in\R_+$, such that $\psi(x)=0, b<x<\infty$, and both $(x^{-1}D)^k x^{-2\mu}\psi(x)$ and $(x^{-1}D)^k x^{-2\mu-1}\psi(x)$  are bounded on $\R_+$. The topology of $\mathcal B_{\mu,b}(\RR_+)$ is generated by the  seminorms
\begin{equation}\label{nsw1}
\beta^\mu_{0}(\psi)=\sup_{ 0\leq k\leq 2\mu}\beta_{0,k}^\mu (\psi),
\mbox{
where }\; \beta_{0,k}^\mu(\psi)=\gamma_{0,k}^{2\mu-1/2}(\psi)+\gamma_{0,k}^{2\mu+1/2}(\psi) \;(m=0).
\end{equation}
\noindent It should be noted that for $b<c$, $\mathcal B_{\mu,b}(\RR_+)\subset \mathcal B_{\mu,c}(\RR_+)$.
For any fixed $\mu\in\N_0$, $\mathcal B^{\mu}(\RR_+)$ denotes the strict inductive limit of $\mathcal B_{\mu,b}(\RR_+)$, where $b$ traverses a monotonically increasing sequence of positive numbers and tends to infinity. Both $\mathcal B_{\mu,b}(\RR_+)$ and $\mathcal B^{\mu}(\RR_+)$ are  complete, see \cite{Zem2}  (cf. \cite[Lemma 2]{Zem1}).
Let
$\mathcal B_{-1/2}(\RR_+)=\bigcap_{\mu\in\N_0}\mathcal B^{\mu}(\RR_+).$
We endow this space with the corresponding projective topology.
The proof of the next lemma is based on  the  proof that $\mathcal B^{\mu}(\RR_+)$ is everywhere dense in $\mathcal K^{\mu}(\RR_+)$.
\begin{lemma}
$\mathcal B_{-1/2}(\RR_+)$ is everywhere dense in $\mathcal K_{-1/2}(\RR_+)$.
\end{lemma}

\begin{proof}
We utilize the proof of \cite[Lemma 6]{Zem2}.

It is obvious that $\mathcal{B}_{-1/2}(\RR_+)$ is a subspace of $\mathcal{K}_{-1/2}(\RR_+)$. Let $\theta\in\mathcal{K}_{-1/2}(\RR_+)$ be such that
$$ \theta(x) = \left\{\begin{array}{ll}
             1, & \mbox{if}\quad -\infty<x<1\\
             0,  & \mbox{if}\quad 2<x<\infty.
         \end{array}
         \right. $$
Let $\theta_r\varphi=\theta(x-r)\varphi$, $r\in\N_0$ is a sequence in  $\mathcal{B}_{-1/2}(\RR_+)$, whenever $\varphi\in\mathcal{K}_{-1/2}(\RR_+)$. Then, $\theta_r$ converges to $\theta$ in $\mathcal K_{-1/2}(\RR_+)$, as $r\rightarrow\infty.$
\end{proof}

Now, we can formulate the Tauberian theorem for the FrHT over $\mathcal K'_{-1/2}(\R_+)$. Note that its formulation is the same as the corresponding result for the elements of $(\mathcal K^\mu(\R_+))'.$ For that reason, we provide the formulation of the theorem but not its proof. Note that a dense set $\mathcal B_{-1/2}(\RR_+)$ is used in the proof.

\begin{theorem}\label {te1} Assume that $f\in \mathcal{K}_{-1/2}'(\R_+)$ is of slow growth at infinity and locally integrable on $\R_+$, defining a regular element of
$\mathcal{K}_{-1/2}'(\R_+)$. Let $\mu_0\in\N_0$, and suppose further that:
\begin{enumerate}
\item [(i)] The limit
\begin{equation}\label{limit1}
\lim_{\varepsilon \to 0^{+} }
\frac{e^{ic_1 (\frac{\xi}{\varepsilon})^2/2} H_{\mu_0}^{\alpha} f (\frac{\xi}{\varepsilon})} {\varepsilon ^{m +1} L(\varepsilon
)}=M_{\xi }<\infty,
\end{equation}

\noindent exists for all $\xi\in\R_+$.

\item [(ii)] There exist constants $N>0$, $C=C(\alpha)>0$, and $0<\varepsilon_0\leq 1$ such that
\begin{equation}\label{ogranicuvanje1}
\frac{|e^{ic_1 (\frac{\xi}{\varepsilon})^2/2} H_{\mu_0}^{\alpha} f (\frac{\xi}{\varepsilon})|} {\varepsilon ^{m+1} L(\varepsilon
)}\leq C |\xi|^{N+\mu_0+\frac{1}{2}},
\end{equation}
\noindent for all $\xi\in\R_+$ and $0< \varepsilon \leq
\varepsilon_0$.
\end{enumerate}
Then $f$ has the quasi-asymptotic behavior of degree $m\in\RR$ at $0^+$ in $\mathcal{K}_{-1/2}'(\RR_+)$.

\end{theorem}

\section{Declaration}
All authors contributed equally to this work, and have given a consent to participate and consent for publication.
No funding was received.

\subsection{Conflict of interest}
All authors declare that they have no conflicts of interest.

\subsection{Data availability}
The authors confirm that the data supporting the findings of this study are available within the article.

\subsection{Acknowledgment}

The first and the fourth authors are supported by the Serbian Academy of Sciences and Arts, project F-10; The second author is supported by the Science Fund of the Republic of
Serbia, $\#$GRANT No. 2727, {\it Global and local analysis of operators and
distributions} - GOALS and by the Ministry of Science, Technological Development and Innovation of the Republic of Serbia (Grants No. 451-03-137/2025-03/ 200169); The first and the third authors are supported by the Republic of Srpska Ministry for Scientific and Technological Development and Higher Education, project \textit{Modern methods of time-frequency analysis and fixed point
theory in spaces of generalized functions with applications}, No. 19.032/431-1-22/23.




\begin{thebibliography}{99}
%

\bibitem{AJMP}   Atanasova, S., Jaksi\'{c}, S., Maksimovi\'{c}, S., Pilipovi\'{c}, S., {Abelian and Tauberian Results for the Fractional Hankel Transform of Generalized Functions}, preprint.

\bibitem{Betancor} Betancor J.J., On Hankel transformable distribution spaces, Publications de L'institut mathematique, 65(79), 123--141 (1999).

\bibitem{Betancor1} Betancor  J.J.  and  Rodriquez-Mesa L., Hankel convolution on distributions spaces with exponential growth, Studia Mathematica 121(1) (1996), 35--52

\bibitem{Hankel1} Bingham N. H., A Tauberian theorem for integral transforms of Hankel type, J. London Math. Soc. (2) 5 (1972).

\bibitem{Gudadhe} Gudadhe A.S., Taywade R.D., Mahalle V.N., Initial and Final Value Theorem on Fractional Hankel Transform, IOSR Journal of Mathematics, 5(1) (2013), 36--39

\bibitem{horvat} Horv\'{a}th J., Topological Vector Spaces and Distributions, Addison-Wesley, 1966.






  \bibitem{5} Kerr, F.H.: Fractional powers of Hankel transforms in the Zemanian spaces. J. Math. Anal. Appl. 166,
65--83 (1992).

   \bibitem{55}  Kerr, F. H., A fractional power theory for Hankel transforms in $L^2(R+)$, J. Math. Anal. Appl. 158 (1991), 114-123.




\bibitem{7}
 Pilipovi\'{c}, S.,  Stankovic B., and  Vindas J.,
 \newblock {\em	 Asymptotic behavior of generalized functions,}
 \newblock World Scientific Publishing  Co.  Pte. Ltd., Hackensack, NJ, 2012.
\bibitem{Pras} Prasad, A., Singh, V.K. The fractional Hankel transform of certain tempered distributions and pseudo-differential operators. Ann Univ Ferrara 59, 141--158 (2013).

\bibitem{Ridenhour} Ridenhour, J. R; Soni, R. P.,  Parseval Relation and Tauberian Theorems for the Hankel Transform, SIAM Journal on Mathematical Analysis; Philadelphia  5(5), (1974), 612--628.

\bibitem{She} Sheppard, C.J.R., Larkin, K.G., Similarity theorems for fractional Fourier transforms and fractional
Hankel transforms. Opt. Commun. 154, 173--178 (1998)

\bibitem{Vindas1}
Vindas, J., Structural Theorems for Quasiasymptotics of
Distributions at Infinity, Pub. Inst. Math. Beograd, N.S.,
84(98)(2008), 159--174.


\bibitem{Vindas2}
Vindas, J., Pilipovi\'{c}, S., Structural theorems for
quasiasymptotics of distributions at the origin, Math. Nachr.
282(2.11)(2009), 1584--1599.


 \bibitem{VDZ} Vladimirov, V.S., Drozinov, Ju., Zavialov, B. I., \textit{Tauberian Theorems for Generalized Functions} Dordrecht: Kluwer Academic Publishers Group, 1988.
\bibitem{Zem}  Zemanian, A.H., Generalized Integral Transformations. Interscience Publishers, New York (1968)

\bibitem{Zem1} Zemanian, A.H., A distributional Hankel transformation, SIAM J. Appl. Math. 14 (1966), 561--576.

\bibitem{Zem2} Zemanian, A.H., The Hankel transformation of certain distributions of rapid growth, SIAM J. Appl. Math. 14 (1966), 678--690.

%
%


\end{thebibliography}
\end{document}